\documentclass[a4paper,11pt,twoside,final]{article}
\usepackage{amsmath,amssymb,verbatim,cleveref,enumerate,cite}
\usepackage{mathptmx}
\usepackage{graphicx}
\usepackage{xcolor}
\usepackage{url,hyperref}
\usepackage{mathptmx}
\usepackage{graphicx}

\newtheorem{theorem}{Theorem}[section]

\newtheorem{corollary}[theorem]{Corollary}

\newtheorem{definition}[theorem]{Definition}
\newtheorem{example}[theorem]{Example}

\newtheorem{lemma}[theorem]{Lemma}

\newtheorem{proposition}[theorem]{Proposition}
\newtheorem{remark}[theorem]{Remark}

\newenvironment{proof}[1][Proof]{\textbf{#1.} }{\ \rule{0.5em}{0.5em}}

\def\a{\alpha}

\def\vp{\varepsilon}
\def\As{\mathcal{A}}
\def\Ps{\mathcal{P}}
\def\G{\Gamma}
\def\dd{\displaystyle}
\def\erre{\mathbb{R}}

\newlength{\wideitemsep}
\let\olditem\item
\renewcommand{\item}{\setlength{\itemsep}{\wideitemsep}\olditem}
\begin{document}
\title{\bf \Large  A multivalued version of the
Radon-Nikod\'{y}m theorem, via the single-valued Gould integral}
\author{Domenico Candeloro \and 
       Anca Croitoru \and
        Alina Gavrilu\c{t} \and
        Anna Rita Sambucini}
\newcommand{\Addresses}{{
  \bigskip

 D.~Candeloro, \textsc{Department of Mathematics and Computer Sciences, University of Perugia,
 1, Via Vanvitelli -- 06123, Perugia (Italy)} 
  \textit{E-mail address}
: \texttt{domenico.candeloro@unipg.it}
https://orcid.org/0000-0003-0526-5334
  \medskip

  A.~Croitoru, \textsc{Faculty of Mathematics, ”Al. I. Cuza” University, 700506 Ia\c{s}i, Romania} 
  \textit{E-mail address}
: \texttt{croitoru@uaic.ro}
 
\medskip

  A.~Gavrilu\c{t}, \textsc{Faculty of Mathematics, ”Al. I. Cuza” University, 700506 Ia\c{s}i, Romania} 
  \textit{E-mail address}
: \texttt{gavrilut@uaic.ro}

\medskip

A.R. ~Sambucini, \textsc{Department of Mathematics and Computer Sciences, University of Perugia,
 1, Via Vanvitelli -- 06123, Perugia (Italy)} 
  \textit{E-mail address}
: \texttt{anna.sambucini@unipg.it}  https://orcid.org/0000-0003-0161-8729
}}
\maketitle

\maketitle \pagestyle{myheadings} \markboth{\centerline{\small \rm   Candeloro, Croitoru, Gavrilu\c{t}, Sambucini}}{\centerline{\small \rm
A multivalued version of the Radon-Nikod\'{y}m theorem...}}
\begin{abstract}
In this paper we  consider a Gould-type integral of 
real functions with respect to a compact and convex valued non necessarily additive measure. 
In particular we will introduce
the concept of integrable multimeasure and, thanks to this notion,  we will establish
 an exact  Radon-Nikod\'{y}m theorem relative to a fuzzy multisubmeasure which is new also in the finite dimensional case. 
Some results concerning the Gould integral are also obtained.
\end{abstract}
\bigskip


\section{Introduction}\label{intro}
Non additive measures are an important field of research in measure theory. 
Due to its applications
 in economics, statistics, human decision making and medicine,
 the field of non-additive measures and of fuzzy measures has been  intensively studied 
in the last years, while  
 the theory of  monotonicity is used in statistics, game theory, probability and artificial intelligence.
(See for example \cite{cope,klmp,bpap}). 
In  \cite{pap2016} 
 Pap has recently studied 
 multivalued integration, examining in particular the Gould integrability for multifunctions and multisubmeasures.
The present research could be connected to his  paper  as a continuation.\\
Concerning with  the theory of integration,
the existence  of a Radon-Nikod\'{y}m derivative is an important tool.
 In fact it provides conditions for the existence of a certain integral representation of measures.
 The Radon-Nikod\'{y}m theorem is used, for example,
for converting  actual probabilities 
into those of the risk neutral probabilities.
Moreover it  was  approached by many authors in several different settings,  
(e.g. \cite{ms-trieste,LM,k2013,k2015,mar1}).
In particular, 
in \cite{CV02} an outline of the previous results is presented, together with quotations 
of the papers in this topic which have appeared since the late 60's.
 Similar problems were studied  afterwards, e.g. in \cite{ms94b,s2} as an extension of \cite{ref11,ref14}, 
later in  \cite{bdpm1,bdpm},
 and also recently deeply examined in  \cite{ckr2013,dpp2015}
both in the countably and the finitely additive case using  different notions of integrals. 
Here we will undertake a similar investigation and we will consider 
fuzzy multisubmeasures defined on an algebra and taking convex compact values  in an arbitrary Banach space $X$.\\
In this paper essentially 
a Radon-Nikod\'{y}m theorem is established,  in order to represent a set-valued additive  measure as the
Gould-type integral of a suitable real-valued function with respect to a fixed fuzzy multisubmeasure.
We point out that  this result is new also in the finite-dimensional case since additivity is requested only for one of the set valued measures involved.  \\
The paper is organized as follows: in Section \ref{basic}
some basic notions and results are given, while in 
 section \ref{tot+G}  some results and examples
regarding Gould type integrability relative to a non-necessarily additive  
measure $m$  are obtained (see also \cite{ccgs2015a,pap2016}). The target space for   $m$
is  the Banach lattice of all real-valued continuous functions defined in a compact, Stonian space 
$\Omega$ (the space $C(\Omega)$). 
This is due to the fact that  Banach lattices  are  often good models for applications, and also for studying set-valued measures
 or functions.
In effect, thanks to the R{\aa}dstr\"{o}m embedding Theorem,  many important hyperspaces 
can be embedded in   $C(\Omega)$ (for example the family of convex and compact subset of a 
Banach space $X$, while for an exhaustive list of such hyperspaces see e.g. \cite{L1} and, 
for various applications in Banach lattices, \cite{ac,s1990a,s1994,cdpms2017}).
In  section  \ref{rnsec} a Radon-Nikod\'{y}m type theorem 
will be obtained (Theorem \ref{trn})  using the set-valued integral defined in \cite{pgc2010,pap2016}. 
According to this result, a multimeasure $\G$ can be expressed as
a Gould type
set-valued integral of a function $f$ with respect to a fuzzy multisubmeasure $M$, that is: 
$\G(E) = \int_{E}fdM,$ for every
$E\in \mathcal{A}$, under  a suitable  exhaustion  condition and the strong absolute continuity
 of $\G$ with respect to $M.$ In this case, 
the construction of the Radon-Nikod\'{y}m derivative makes use of  the mentioned notion of exhaustion,
introduced by Maynard  \cite{ref14} in the scalar case and extended by other authors to the vector valued case: 
\cite{ref11,ms94b,s2}.
As an application of the Theorem \ref{trn} an integration by substitution theorem is obtained for fuzzy multimeasures.
\section{Basic facts and definitions.}\label{basic}
Unless stated otherwise, throughout this paper $T$ is an abstract nonvoid set, $\mathcal{P}(T)$
 the family of all subsets of $T$,
$\mathcal{A}$ an algebra of subsets of $T$ and $\mu:\mathcal{A}\rightarrow [0,+\infty)$ 
an arbitrary set function,
with $\mu(\emptyset)=0$.
 A \textit{partition} of $T$ is a finite family of nonvoid sets 
$P=\{A_{i}\}_{i=1}^{n}\subset \mathcal{A}$ such that 
$A_{i}\cap A_{j}=\emptyset ,i\neq j$, and $\bigcup_{i=1}^{n}A_{i}=T.$
  Let $P=\{A_{i}\}_{i=1}^{n}$ and $P^{\prime }=\{B_{j}\}_{j=1}^{q}$ be two partitions of $T$.
The partition $P^{\prime }$ is said to be \textit{finer than} $P$, denoted by $P\leq
P^{\prime }$ (or, $P^{\prime}\geq P$), if for every $j\in \{1,\ldots,q\}$, there exists $i_{j}\in\{1,\ldots,n\}$ so that 
$B_{j}\subseteq A_{i_{j}}$.
 The \textit{common refinement} of two partitions 
$P=\{A_{i}\}_{i=1}^{n}$ and $P^{\prime }=\{B_{j}\}_{j=1}^{q}$ is the partition 
$P\vee P^{\prime }=\{A_{i}\cap B_{j}\}_{ i \in\{1,\ldots,n\},j \in\{1,\ldots,q\}}$.
Obviously, 
$P\vee P^{\prime }\geq P$ and $P\vee P^{\prime }\geq
P^{\prime }.$
The class of all partitions of $T$ will be denoted by $\mathcal{P}$, and if
 $A\in  \mathcal{A}$ is fixed, $\mathcal{P}_{A}$ denotes the class of all partitions of $A$.
Given $\mu$, we will consider $\overline{\mu}, \mu^*:\mathcal{P}(T)\to [0,+\infty],$  \textit{the variation, semivariation} of $\mu$ respectively and $\widetilde{\mu} :\mathcal{P}(T)\to [0,+\infty]$, 
given by
$\widetilde{\mu}(E)=\inf \{\overline{\mu}(A);E\subseteq A, A\in \mathcal{A}\}$. For the properties of variation, semivariation and $\widetilde{\mu}$ see for example  \cite{gcmed}.
\\
Let  $(X,\|\cdot\|)$ be a Banach space, $B_X$ its unit ball;
the symbol $m$ will be used for vector valued set functions.
For a vector measure $m:\mathcal{A}\to X$, its semivariation $m^{*}: \mathcal{P}(T)\to [0,+\infty]$ is defined   by:
$m^{*}(E) = \sup \{\|m(A)\|; A\in \mathcal{A}, A\subseteq E\}.$
In an analogous way we can define 
the {\em variation}  $\overline{m}:= \overline{\|m \|}$. 
Thus, if $A\in \mathcal{A}$, then $\|m(A)\|\leq \overline{m}(A)$, which implies that $m^{*}(E)\leq \overline{m}(E)$,
 for every $E\in \mathcal{P}(T)$.\\
Let 
$ck(X)$ be  the family of all
nonempty compact convex subsets of a real Banach space $(X,\|\cdot\|)$.
By the symbol $+$  the  Minkowski addition  will be indicated.
Let $h$ be the Hausdorff metric on $ck(X)$.
 It is well-known that
 $(ck(X),h)$ is a complete
metric space (see for example \cite[Theorem II-14]{cv}). 
Finally, for any bounded set $A$,   
 $|A|_h$ denotes the distance  $h(A,\{0\})$, where $0$ is the origin
of $X$.
With the symbol $M$ we denote  a set function with values in $ck(X)$.
Now, several notions   are recalled for further use.


\begin{definition}\label{ex-null} 
\rm (\cite{ref11}, \cite[Definition 3.2]{ms-trieste})
Let $\mu:\mathcal{A}\to \overline{\mathbb{R}}_{+}$ be finitely additive.
\begin{description}
\item[\rm \bf \ref{ex-null}.a)]  A finite or countable family of pairwise disjoint sets
$(E_{i})_{i\in I}\subset \mathcal{A}$ will be called a $\mu$-\textit{ exhaustion} of 
$E \in \mathcal{A} $\, 
if $\mu(E_{i})>0$ for every $i\in I$ and for each $\varepsilon>0$, there is 
$n_{0}(\varepsilon) = n_{0}\in \mathbb{N}$ such
that $\mu(E \backslash \bigcup\limits_{i=1}^{n_{0}}E_{i})<\varepsilon.$
\item[\rm \bf \ref{ex-null}.b)] A set property $P$ is said to be $\mu$-\textit{exhaustive} on
$E\in \mathcal{A}$ if there exists a $\mu$-exhaustion $(E_i)_i$ of
$E$, such that every $E_i$ has $P$.
\item[\rm \bf \ref{ex-null}.c)] A set property $P$ is called $\mu$-\textit{null difference} if
whenever $A,B\in \mathcal{A}$ with $\mu(A)>0$ and $\mu(B)>0$, from
$\mu(A \triangle B)=0,$ it follows that either $A$ and $B$ both
have $P$ or neither does.
\item[\rm \bf \ref{ex-null}.d)]  A property (P) about the points of $T$ \textit{ holds $\widetilde{\mu}$-almost everywhere}
 (denoted \textit{ $\widetilde{\mu}$-a.e.}) if there exists $A\in\mathcal{P}(T)$ so that 
$\widetilde{\mu}(A) = 0$ and (P) holds on $T\backslash A$.
\end{description}
\end{definition}

\noindent For an arbitrary {real} function \mbox{$f:T\to \mathbb{R},$} the symbol  
$\sigma _{m}(f,P)$ (or, if there is no doubt, $\sigma(f,P)$,     
 $\sigma _{m}(P)$ or $\sigma (P)$) denotes the sum 
$\sum_{i=1}^{n}f(t_{i})m(A_{i}),$ for every partition of $T$, $P=\{A_{i}\}_{i=1}^{n}$ 
and every $t_{i}\in A_{i},i\in \{1,\ldots,n\}$. With the same meaning we define $\sigma_M (f,P)$, for non negative
 $f$ and $ck(X)$-valued $M$.

\section{Gould integral.}\label{tot+G}
We now introduce  the definition 
of Gould integrability.
The Gould integral was defined in \cite{Gould} for real functions with respect to a finitely additive vector 
measure taking values in a Banach space. Different generalizations and topics were introduced and studied in \cite{gcmed,ref7,ref6,pgc2010,ccgs2015a,pap2016}.
\\
Moreover, since we want to study and consider mainly the multivalued case,
(i.e. set functions taking values in some space of bounded convex sets)
 we focus our attention on the Banach space  $(C(\Omega),\|\cdot\|_{\infty})$.
 This is due to the fact that, thanks to the R{\aa}dstr\"{o}m Embedding Theorem, many
 important hyperspaces can be embedded in   $C(\Omega)$ (for a list of such hyperspaces see e.g. \cite{L1}).
We remember moreover that $C(\Omega)$ is also a Banach lattice in which 
the symbol $| \cdot |$  denotes  the modulus. 
  So, rather than considering a general Banach space $(X, \|\cdot\|)$, or a Banach lattice, from now on we restrict ourselves to
mappings $m:\mathcal{A} \to C(\Omega)$,
with $\Omega$ compact, Hausdorff and we can give 
 the notion of subadditivity for $C(\Omega)$-valued set functions in the usual way:  $m(\emptyset) = 0$ and 
$m(A\cup B)\leq m(A)+m(B)$
holds, when $A,B\in \As$, $A\cap B=\emptyset$

\begin{definition}\label{g}\rm A real function $f:T\to \mathbb{R}$ is said to be
\begin{enumerate}[ \rm \bf \ref{g}.a)]
 \item   (\textit{Gould}) $m$\textit{-integrable}
\textit{on }$T$ if the net $(\sigma (P))_{P\in (\mathcal{P},\leq )}$ is
convergent in $C(\Omega)$, where $\mathcal{P}$ is ordered by the relation $"\leq"$. 
 If $(\sigma (P))_{P\in (\mathcal{P},\leq )}$ is convergent, then its limit
is called \textit{the Gould integral of }$f$\textit{\ on }$T$\textit{\ with
respect to }$m$, denoted by $(G) \int_{T}fdm$ (shortly $\int_{T}fdm$).
\item  $m$\textit{-integrable on }$B \in \mathcal{A}$\, if
the restriction $f|_{B}$ of $f$ to $B$ is $m$-integrable on $(B,\mathcal{A}_{B},m_{B})$.
\end{enumerate}
\end{definition}

\begin{remark}\rm\label{somma}
 Thus $f$ is $m$-integrable on $T$ if and only if there
exists $\a\in C(\Omega)$ such that for every $\varepsilon >0$, there exists
a partition $P_{\varepsilon }$ of $T$, so that for every other partition of $T$, $P=\{A_{i}\}_{i=1}^{n},$ with 
$P\geq P_{\varepsilon }$ and
every choice of points $t_{i}\in A_{i},i\in\{1,\ldots,n\}$,  one has $\|\sigma (P)-\a\|_{\infty}<\varepsilon $.
Moreover 
If $f_1,f_2$ are $m$-integrable and $\alpha$ is any real constant, then $\alpha f_1$ is $m$-integrable, $f_1+f_2$ is $m$-integrable, and the integral is 
linear.
\end{remark}

\begin{proposition}\label{gouldsotto}
Let 
$f:T\to \mathbb{R}$ be any Gould-integrable mapping with respect to $m$. Then, if $A$ is any fixed element of 
$\As$, the mapping $f \mathbf{1}_A$ is integrable too.
\end{proposition}
\begin{proof}
Denoting by  $\Pi_A$ the partitions of the set $A$, it is not difficult to prove 
 that the sums $\sigma(f,\Pi_A)$ satisfy a Cauchy principle in $C(\Omega)$;
 since this space is complete with respect to its norm, the assertion follows.
\end{proof} \\


\begin{example}\label{ex}\rm Some examples of  Gould integrable functions with respect to $m$ are given here:
\begin{enumerate}[\rm  \ref{ex}.a)]
\item Let $T$ be a finite set, $\mathcal{A}=\mathcal{P}(T)$, $m:\mathcal{A}\to C(\Omega)$ and
 $f:T\to \mathbb{R}$ be arbitrary.
 Then $f$ is Gould $m$-integrable and $\int_{T}fdm = \sum\limits_{t\in T}f(t)m(\{t\})$.
\item If $m:\mathcal{A}\to C(\Omega)$ is finitely additive and $f:T\to \mathbb{R}$ is simple, $f= 
\sum{i=1}^{n} a_{i}\cdot 1_{A_{i}}$, then $f$ is Gould $m$-integrable and $\int_{T}fdm = 
\sum_{i=1}^{n} a_{i}\cdot m (A_{i})$.
\end{enumerate}
\end{example}

Moreover
the previous example \ref{ex}.b) can be improved as follows.
\begin{proposition}\label{funzionidiscrete}
Let $\As$ be a $\sigma$-algebra, 
and $m:\As\to C(\Omega)$ be finitely additive, and  assume that $(A_n)_{n\in \mathbb{N}}$ is a countable family of pairwise disjoint
 elements of $\As$, such that $\lim_n \overline{m}(\cup_{j>n}A_j)=0$. Then, the function $f:T\to \mathbb{R}$ defined as
 $f=\sum_n c_n1_{A_n}$ is Gould-integrable as soon as the sequence $(c_n)_n$ is 
bounded in $\mathbb{R}$; in this case, 
$\int_T f dm =\sum_n c_n m(A_n)$.
\end{proposition}
\begin{proof}
Under these assumptions, it is clear that the real-valued series
$\sum_n |c_n| \overline{m}(A_n)$
is convergent, hence the series $\sum_n c_n m(A_n)$ is convergent in $C(\Omega)$.
We will show that $f$ is integrable and its integral coincides with $\sum_n c_n m(A_n)$.
Define now $S:=\bigcup_n A_n$, and fix $\varepsilon>0$. Then there exists $N\in \mathbb{N}$ such that 
$\overline{m}(\cup_{j>N}A_j)<\varepsilon$. Therefore 
$\sum_{j>N}|c_j|\overline{m}(A_j)\leq K\varepsilon$
where $K$ is any bound for $|c_n|$, \ $n\in \mathbb{N}$. 
 Now set
$F:=\bigcup_{j\leq N}A_j,$
and choose any partition $P$ of $T$, finer than $\{F,S\setminus F, T\setminus S\}$. Setting 
$P=\{(B_i,t_i),i=1..,k\}$, one 
 then  has
$\sigma(f,P) :=\sum_{i=1}^kf(t_i)m(B_i)=\sum_{i\in I_1}f(t_i)m(B_i)+\sum_{i\in I_2}f(t_i)m(B_i),$
where $I_1=\{i:B_i\subset F\}, \ I_2=\{i:B_i\subset S\setminus F\}$.\\
Of course
$\sum_{i\in I_1}f(t_i)m(B_i)
=\sum_{j=1}^Nc_j m(A_j)$
and
$\|
\sum_{i\in I_2}f(t_i)m(B_i)\|_{\infty}\leq K \overline{m}(\cup_{j>N}A_j))\leq K\varepsilon.$\\
So,
$
\|\sigma(f,P)  -\sum_n c_n m(A_n) \|_{\infty}\leq\| \sum_{i\in I_2} c_j m(A_j)\|_{\infty}+
\sum_{j>N}
|c_j| \overline{m}(A_j)\leq 2K \varepsilon.$
This concludes the proof.
\end{proof} \\

For more general functions, proceeding as in the proof of \cite[theorem 1.4]{mimmoroma}
 and \cite[Proposition 6]{SISY}, one can deduce the following proposition and the subsequent corollary.
In this situation the absolute value replaces the norm of $C(\Omega)$.
\begin{proposition}\label{henlemma} 
Let $f:T\to \mathbb{R}$ be any integrable function. Then there exists a sequence $(\Pi_n)_n$ of partitions
 such that, for every $n$ it is
$\sum_{E\in \Pi_n}Ob(f, E)\leq \dfrac{u}{n},$
where $Ob(f,E)=\sup_{\Pi'_E,\Pi''_E} \left\{\left| \sum_{F''\in \Pi''_E}f(t)m(F'') - 
\hskip-2mm\sum_{F'\in \Pi'_E}f(s)m(F')\right|,
\, \forall\,\, t \in F'', s \in F' \right\},$
 and $\Pi'_E, \Pi''_E$ run along all  partitions of $E$.

\end{proposition}
\begin{proof}
First observe that, thanks to the Cauchy criterion, a sequence $(\Pi_n)_n$ of partitions   exists,  such that, 
for every integer $n$
\begin{eqnarray}\label{primocauchy}
\left|\sum_{F'\in \Pi'}f(s)m(F')-\sum_{F''\in \Pi''}f(t)m(F'') \right|\leq \frac{u}{n}
\end{eqnarray}
(with obvious meaning of symbols)
holds, for all partitions $\Pi', \ \Pi''$ finer than $\Pi_n$. Now, take any integer $n$ and, for each element 
$E$ of $\Pi_n$, 
consider two arbitrary partitions $\Pi'_E$ and $\Pi''_E$ of $E$.
Then, taking the {\em union} of the partitions $\Pi'_E$ as $E$ varies, and making the same operation with
the partitions $\Pi''_E$,  two partitions of $T$ are obtained, finer than $P_n$, for which (\ref{primocauchy})
 holds true. From  (\ref{primocauchy}),  obviously it follows
\begin{eqnarray}\label{secondocauchy}
\sum_{F'\in \Pi'}f(s)m(F')- \sum_{F''\in \Pi''}f(t)m(F'')\leq \frac{u}{n}.
\end{eqnarray}
Now,   let $E_1$ be the first element of $\Pi_n$.
In the summation at left-hand side, fix all the $F's$ and the $F''s$ that are not contained in $E_1$. 
Taking the supremum when the remaining $F's$ and $F''s$  vary in all possible ways, it follows
\begin{eqnarray*}
&& \sup_{\Pi'_{E_1}}\sigma(f,\Pi'_{E_1})-\inf_{\Pi''_{E_1}}\sigma(f,\Pi''_{E_1})+
 \sum_{\substack{F'\in \Pi', \\ F'\not\subset E_1}}f(s)m(F')-
 \sum_{\substack{F''\in \Pi'', \\ F''\not\subset E_1}} f(t)m(F'')
\leq \frac{u}{n},
\end{eqnarray*}
namely
\begin{eqnarray*}
 Ob(f,E_1)+ \hskip-2mm \sum_{\substack{F'\in \Pi',\\ F'\not\subset E_1}}f(s)m(F') - \hskip-2mm
 \sum_{\substack{F''\in \Pi'',\\F''\not\subset E_1}}f(t)m(F'')\leq \frac{u}{n}.
\end{eqnarray*}
In the same way , fixing all the $F'$ and $F''$ that are not contained in the second subset of $\Pi$, (say $E_2$), 
and making the same operation, it follows 
\begin{eqnarray*}
\hskip-1cm&& 
Ob(f,E_1)+Ob(f,E_2)+\hskip-2mm \sum_{\substack{F'\in \Pi',\\ F'\not\subset E_1\cup E_2}}f(s)m(F')  
- \sum_{\substack{F''\in \Pi'',\\F''\not\subset E_1\cup E_2}}f(t)m(F'')\leq \frac{u}{n}
\end{eqnarray*}
Now, it is clear how to deduce the assertion.
\end{proof} \\

 Concerning the previous Proposition, we remark that, unless the space $X$ is finite-dimensional, a similar conclusion fails to hold if the absolute value is replaced by the norm: this is noteworthy if one considers its consequences, in particular the corollary \ref{4.9}.  

The following result states an easy consequence of  Proposition \ref{henlemma} and it  can be viewed 
as a Henstock Lemma result.
\begin{corollary}\label{henstoc2}
Let $g:T\to \erre$ be any mapping, then $g$ is Gould-integrable  if and only if  there exists a sequence $(\Pi_n)_n$ of partitions, 
such that, for every 
$n$ and every partition $\Pi$ finer than $\Pi_n$
\begin{description}
\item[\ref{henstoc2}.1)] 
$\sum_{E\in \Pi} \left|g(\tau_E)m(E)-\int_E g dm \right| \leq \frac{u}{n},$
where $\tau_E$ is any point in the set $E$.
\end{description}
\end{corollary}
\begin{proof}
The  "if part"  is a consequence of Proposition \ref{henlemma}.
\end{proof} \\

Now we want to focus our attention on a particular type of set valued mappings $m$ which will be useful 
in the last section, i.e. the {\em Gould-integrable} ones. 
 A similar notion was also  given  in \cite[Definition 1.1]{mimmoroma}, though for set functions taking 
values more generally in a vector lattice.
Notice that we will  use the symbol $\mathbf{1}$ to denote the real-valued function on $T$,
defined by $\mathbf{1} (t) \equiv 1$, while the symbol $u$ denotes the element of $C(\Omega)$ 
constantly equal to 1.
We remember also that it is well-known that the norm $\|\cdot\|_{\infty}$ coincides with the {\em unit} 
norm $\|\cdot\|_u$.

\begin{definition}\label{minteg}\rm 
Given a mapping $m:\mathcal{A}\to C(\Omega)$, such that $m(\emptyset)=0$, 
$m$ is said to be {\em Gould-integrable} if the  mapping $\mathbf{1}: T \to \mathbb{R}$ 
 is Gould integrable with respect to $m$. 
We denote by $\upsilon_m(T):= \int_T \mathbf{1} dm$ its integral.
\end{definition}

By  Proposition \ref{gouldsotto}, if $m$ is Gould integrable, then $m$ is integrable in every 
measurable set $A\subset T$. Moreover, 
denoting by $\upsilon_m(A)$ the integral of $m$ in $A$,
 the mapping $A\mapsto \upsilon_m(A)$ is  finitely additive, as will be proved in the Proposition \ref{additivita0}.
In other words, $m$ is Gould-integrable if and only if there exists  $\upsilon_m : \mathcal{A}\to C(\Omega)$
such that, for every set $A\in \As$ and for every $\varepsilon>0$ a partition $P\in \Ps$ can be found, such that
$\left\|\sum_{I\in P'}m(I\cap A)-\upsilon_m(A) \right\|_{\infty}\leq \varepsilon$
holds, as soon as $P'$ is finer than $P$.
When this is the case, then $\upsilon_m$ 
is called the {\em integral function} of $m$.\\

Examples of non-additive set functions that are Gould integrable could be the following:
\begin{example}\label{ex-int} \rm 
Let  $T=[0,1]$ endowed with the usual Borel $\sigma$-algebra $\Sigma$ and Lebesgue measure 
$\lambda$.  
\begin{enumerate}[\rm \ref{ex-int}.a)]
\item
Let 
$m(A) = \lambda^2 (A) \cdot u$
for every $A\in \Sigma$. Clearly $m$ is not additive (it is superadditive), but it has null integral:
 indeed, for any $\varepsilon>0$ take any partition $P$ of $[0,1]$ consisting of pairwise disjoint 
measurable sets $A_i$, each with measure less than $\varepsilon$. Then
$\|\sum_i m(A_i)\|_{\infty}=\sum_i \lambda(A_i)^2\leq \sum_i\lambda(A_i)\varepsilon=\varepsilon.$
Of course, the same happens for every finer partition than $P$.
\item
Let $\gamma(A) = (\lambda(A) - \lambda^2 (A)) \cdot u$, 
then
$\gamma$ is non additive (it is subadditive)  and integrable too.
\item 
 Let $X$ be any
finite-dimen\-sional Banach space, whose unit ball is denoted by $B_X$. Let  $(W_t)_t$ denote the
 standard scalar
 Brownian motion, $t\in [0,T^{\star}]$, and set
$ (B_t)_t = (W_t B_X )_t , \, 
t \in [0,T^{\star}].$
 This clearly defines a set-valued process.
If $U$ denotes the R{\aa}dstr\"{o}m  embedding of
the family of compact and convex subsets of $X$  
into $C(\Omega)$, as we will 
recall 
in Theorem \ref{LABU}, then $U(B_X)=u$, 
where $u$ is the element of $C(\Omega)$ constantly equal to 1. Therefore, 
\(t\mapsto W_t u\)  defines a $C(\Omega)$-valued process.
Now, let $\mathcal{A}$ be the algebra in $[0,T^{\star}]$ generated by all (half-open) subintervals, and define 
$m: \mathcal{A}\to C(\Omega)$ 
in the following way: 
\[ m(A) = \left\{ \begin{array}{ll}
(W_b-W_a)^2 \cdot u & A=]a,b] \\
\sum_i(W_{b_i}-W_{a_i})^2 \cdot u &
A \mbox{ is the finite union of (maximal) disjoint intervals }\\ & ]a_i,b_i].
\end{array} \right. \]
Then, for any partition $P$ of $[0,T^{\star}]$, $P:=\{I_1,...,I_k\}$ into pairwise disjoint elements 
of $\mathcal{A}$, define
$S(P)=\sum_{j=1}^k m(I_j)$
and observe that, thanks to well-known properties of the Brownian Motion, this quantity tends to 
$T^{\star} \cdot u$ in $L^2$ when the maximum length of the partitions tends to 0. Therefore, 
at least for this type of convergence, the measure $m$ has integral $T^{\star} \cdot u$,
 and in every interval $[a,b]\subset [0,T^{\star}]$ the integral is $(b-a) \cdot u$.
\end{enumerate}
\end{example}
\begin{proposition}\label{additivita0}
If $m$ is Gould integrable then its integral function 
 $\upsilon_m$, defined in $\As$ as
$\upsilon_m(A)=\int_T \mathbf{1}_A\, dm$, is additive.
\end{proposition}

\begin{proof}
It follows immediately from the Remark \ref{somma}.
\end{proof} \\

So the Gould integrability of $m$ allows to link $m$ with $\upsilon_m$ which is an additive set function and  
 clearly, $m$ is additive if and only if it is integrable and $m=\upsilon_m$.
 Moreover, for bounded functions, the following  characterization can be given:
\begin{corollary}\label{4.9}
Assume that $m$ 
is integrable. Then a bounded function $f:T\to \mathbb{R}$ is Gould-integrable with respect to $m$ if 
and only if it is with respect to $\upsilon_m$, and the two integrals agree.
\end{corollary}
\begin{proof}
Assume that $f$ is integrable with respect to $m$, and denote by $K$ any majorant for $|f|$. Now, fix arbitrarily
 $\varepsilon>0$: correspondingly, there exists a partition $P_1$ such that 
\begin{eqnarray*}
\left\|\sum_{I\in P}f(t_I)m(I)-\int_T f dm \right\|_{\infty} \leq \varepsilon\hskip.7cm
{\rm }\ i.e.\hskip.7cm \left|\sum_{I\in P}f(t_I)m(I)-\int_T f dm \right| \leq \varepsilon u
\end{eqnarray*}
holds, for every partition $P$, finer than $P_1$.
Let $n$ be such that $1/n \leq \varepsilon$,
 by the Corollary \ref{henstoc2}, for $g=1$, there exists also a partition $P_2$ such that
$\sum_{E\in \Pi}|m(E)-\upsilon_m(E)|\leq  \varepsilon u$
holds, for every partition $\Pi$ finer than $P_2$.
So, if $P$ is any partition finer than $P_1\vee P_2$, one gets
\begin{eqnarray*}
&&  \hskip-.8cm \left|\sum_I f(t_I)\upsilon_m(I)-\int_T f dm \right| \leq
\left| \sum_I [f(t_I)\upsilon_m(I)-f(t_I)m(I)]  \right| + 
\left|\sum_I f(t_I)m(I)-\int_T f dm  \right| \leq\\
&&  \hskip-.8cm \leq \sum_I \left| f(t_I)\left(\upsilon_m(I)-m(I)\right) \right| + \vp u
 \leq
 K\sum_I |m(I)-\upsilon_m(I)|+ \vp u\leq (1+K)\varepsilon u.
\end{eqnarray*}
So
$ \left\|\sum_I f(t_I)\upsilon_m(I)-\int_T f dm \right\|_{\infty} \leq (1+K) \varepsilon.$
This clearly suffices to conclude that $f$ is integrable with respect to $\upsilon_m$ and the two integrals agree.
A similar argument can be used to prove also the reverse implication. Hence the proof is finished. 
\end{proof} \\

We remark that, for bounded functions, the Corollary \ref{4.9} allows to deal the non-additive case by means of the additive one, similarly as the Stone estension Theorem which connects  $L^1(m)$, when $m$ is finitely additive, with $L^1(\nu)$, where $\nu$ is the countably additive {\em transform} of $m$.

We can observe that the results obtained in Propositions \ref{gouldsotto}, \ref{funzionidiscrete} and \ref{additivita0} 
 are still valid in an arbitrary Banach space and not only in $C(\Omega)$ and 
we remember also that   notions of order-type integrals have also been investigated, for functions 
taking their values in ordered 
vector spaces, and in Banach lattices: see for example \cite{ref6,SISY,bcs2014}.
\section{A Radon-Nikod\'{y}m type theorem}\label{rnsec}

This section deals with a Radon-Nikod\'{y}m type theorem for multimeasures using the notion of exhaustion, 
following a method
 of Maynard  \cite{ref14,ref11,ms94b,
s2}.  
We recall that the Hausdorff distance $h$ is defined by
$h(A,B)=\max\{e(A,B),e(B,A)\},$
where the {\em excess} $e(A,B)$ is defined as
$e(A,B):=\sup_{a\in A}d(a,B):=\sup_{a\in A}  \inf_{b \in B} \| a-b\|.$
In particular
\begin{remark} \rm \label{rem_H}
If $A\subset B$,  then $e(A,B)=0$ and $h(A,B)=e(B,A)$.
Moreover, observe that for any non-empty bounded set $A\subset X$, and any pair $(t,b)$ of elements of $X$, 
$|d(t,A)-d(b,A)|\leq \|b-t\|.$
Indeed, let $\sigma:=\|b-t\|$ and fix arbitrarily $\vp>0$. Then there exists $a\in A$ such that
 $d(t,A)\geq d(t,a)-\vp\geq d(b,a)-\sigma-\vp\geq d(b,A)-\vp-\sigma.$
By the arbitrariness of $\vp$, it follows that $d(t,A)-d(b,A)\geq -\sigma$, i.e.
$d(b,A)-d(t,A)\leq \sigma.$
Exchanging the roles between $t$ and $b$, one obtains
$d(t,A)-d(b,A)\leq \sigma$
and therefore $|d(t,A)-d(b,A)|\leq \|b-t\|$.
Another useful fact is the following:
 for every pair of bounded subsets $A,B\subset X$,  $e(B,A)=e(cl(B),A).$
Of course, since $B\subset cl(B)$, it is clear that $e(B,A)\leq e(cl(B),A)$.\\
 Viceversa, fix $\varepsilon>0$: 
then  $j\in cl(B)$ exists, such that $d(j,A)\geq e(cl(B),A)-\vp/2$. Now, let $h\in B$ be such 
that $\|h-j\|\leq \vp/2$: then $|d(h,A)-d(j,A)|\leq \vp/2$, and so 
$e(B,A)\geq d(h,A)\geq d(j,A)-\vp/2\geq e(cl(B),A)-\vp.$
By the  arbitrariness of $\vp$, this gives the reverse inequality
and the proof is complete.
\end{remark}

By \cite[Proposition 1.19 Chapter 7]{Hu} we have that 
\begin{proposition}\label{riparazione}
Let 
$(A_n)_n$ be any increasing sequence of  compact convex subsets of $X$, and assume that a 
compact convex set $K$ exists, such that $A_n\subset K$ for all $n$.
Then $\lim_n h(A_n,J)=0$,
where $J:=cl(\bigcup_n A_n)$.
\end{proposition}

For set valued functions we recall from   the following concept: 
\begin{definition}\label{msm}\rm
A set function
 $M:\mathcal{A}\to ck(X)$ is said to
be a
\textit{multisubmeasure} if:
$M (\emptyset )=\{0\}$ and 
$M (A\cup B)\subset M(A) + M(B),$ 
for every $A,B\in \mathcal{A},$ with $A\cap B=\emptyset$ 
(or, equivalently, $M (A\cup B)\subset M (A) + M
(B),$ for every $A,B\in \mathcal{A}$).\\
$M$ is said to be a 
 \textit{fuzzy multisubmeasure}  if moreover:
$M (A)\subset M (B),$ for every $A,B\in \mathcal{A},$ with
$A\subset B$ (that is, $M$ is \textit{ monotone {\rm on}
$\mathcal{A}$}).
If
$M (A\cup B)\subset M(A) + M(B),$ 
for every $A,B\in \mathcal{A},$ with $A\cap B=\emptyset$
then $M$ is said to be a \textit{multimeasure} that is, $M$
 is finitely additive.
\end{definition}
Examples of   fuzzy multisubmeasures $M$ are given in \cite{ref6}, moreover we can consider also
 $M(A) = [0, \lambda(A) - \lambda^2(A)] \cdot u$, where $\lambda$ and $u$ are as given in Example 
\ref{ex-int}.

\begin{definition}\label{bv}\rm
Let $M :\mathcal{A}\to ck(X)$ be a multivalued set function, with $M(\emptyset)=\{0\}$. 
Consider the following set functions associated to $M$:
\begin{enumerate}[\bf \ref{bv}.a)]
\item $|M(\cdot)|_h$ defined by $|M(A)|_h=h(M(A),\{0\})=\sup\{\|x\|: x\in M(A)\}$ for every $A\in \As$.
\item $v_M(\cdot)$ defined by $v_M(A)=\sup \{\sum\limits_{i=1}^{n}|M(E_{i})|_h\},$ 
for every $A\in \mathcal{A},$ where the supremum is extended over all finite partitions $\{E_{i}\}_{i=1}^n$ of $A.$
$v_M(\cdot)$ is said to be \textit{the variation} of $M $. 
$M$ is said to be \textit{of finite variation} if $v_M(T)<\infty $.
\end{enumerate}
\end{definition}

In the sequel, let $M :\mathcal{A}\to ck(X)$ be a  fuzzy multisubmeasure  
and $f$ 
 a non negative real-valued function. 
Let $\sigma(P)=\sigma_{f,M}(P) =\sum_{i=1}^{n}f(t_{i}) M (A_{i})$,
for every partition $P=\{A_{i}\}_{i=1, \ldots, n}$ of $T$ and every $t_{i}\in A_{i},i=1,\ldots,n$. Then
\begin{definition}\label{gm}\rm 
$f$ is said to be $M $-\textit{integrable (on $T$)} if the net 
$(\sigma (P))_{P\in (\mathcal{P},\leq )}$ is convergent in $(ck(X),h)$,
where $\mathcal{P}$ is the set of all partitions of $T$ and
$"\leq "$ is the order relation on $\mathcal{P}$ given in
Definition \ref{g}.a).
 Its limit is called \textit{the integral of $f$\ on\ $T$} with respect to the fuzzy multisubmeasure $M$ and is denoted by
$\int_{T}f dM$.\\
 If $B\in \mathcal{A},$ then $f$ is said to be $M $-\textit{integrable on $B$} if the restriction 
$f|_{B}$ of $f$ to $B$ is $M $-integrable on 
$(B,\mathcal{A}_{B},M _{B})$.
\end{definition}

In other words, $f$ is $M$-integrable in $T$ if there exists an element $J\in ck(X)$, such that for every
 $\varepsilon>0$ 
there exists  a partition $P\in \mathcal{P}$ with the property that $h(\sigma(P'),J)\leq \varepsilon$ holds true,
 for every partition $P'$ finer than $P$.\\

As well highlighted in \cite{L1} the space $ck(X)$ is a sub-near vector lattice of  $cwk(X)$ (non empty, weakly compact and convex subsets of $X$) with respect to the operations of additions and multiplication by positive scalars and to order induced by $cwk(X)$; moreover if $X$ is not finite dimensional, this hyperspace can be considered as a subset of $S_1 =cbf(X)$ (non empty, convex, closed, bounded subset of $X$) and 
it can be embedded, using the structure  of $S_1$, provided that 
$u=B_X$, $\vec{0}= \{0\}$, in such a way that the norm of the embedding space is a Riesz norm. So, using Kakutani's M-space representation theorem, the near vector lattice $ck(X)$ with order units, endowed with the Hausdorff metric can be represented in terms of $C(\Omega)$ spaces, as shown in:

\begin{theorem}\label{LABU}{\rm(\cite[Theorem 5.7]{L1}).}
Let $X$ be any Banach space. Then there exist a compact, stonian,  Hausdorff space $\Omega$ and an
isometry   $U:ck(X) \to C(\Omega)$ such that
\begin{enumerate}[\rm \bf  \ref{LABU}.1)]
\item $U(\alpha A+\beta C)=\alpha U(A)+\beta U(C)$ for all $\alpha,\beta\in \mathbb{R}^+$ and 
 $A,C\in ck(X)$.
\item $h(A,C)=\|U(A)-U(C)\|_{\infty}$ for all  $A,C\in ck(X)$.
\item$U(ck(X))$ is norm-closed in $C(\Omega)$.
\item $U(\overline{co}(A\cup C))=\max\{U(A),U(C)\}$,  for all $  A,C\in ck(X)$.
\end{enumerate}
\end{theorem}
Observe now that the
 embedding theorem  can be used in order to replace the multivalued integral above 
with a single-valued one, at least for {\em positive} integrands $f$.  
This leads to the following 
\begin{definition}\label{immersa}\rm  
Define
\mbox{$U_M:\mathcal{A}\to C(\Omega)$} as 
$U_M(E)=U(M(E))$ for all $E\in \As$. The mapping $U_M$ will be called the {\em embedded} mapping of $M$. Moreover, thanks to
{\bf \ref{LABU}.4}), the embedded mapping $U_M$ is a fuzzy submeasure if $M$ is a fuzzy multisubmeasure. 
\end{definition}

Thanks to the Theorem \ref{LABU}, it is clear that $\|U_M(E)\|_{\infty}=|M(E)|_h$ for every $E\in \mathcal{A}$, 
and so $M$ is of bounded variation if and only if $U_M$ is, as a $C(\Omega)$-valued set function.\\

Since we can consider also Gould-integrability with respect to 
 $U_M$ (according to Definition \ref{g})  for mappings $f:T\to \mathbb{R}^+_0$,
then the following result holds:
\begin{theorem}\label{singlegould}
A function $f$ is
$M$ integrable, 
with integral $J$, if and only if it is Gould-integrable with respect to $U_M$ and its integral is $j$.
Then the two integrals satisfy: $ U(J)=j$.
Finally, in these cases, $f\mathbf{1}_A$ is integrable for every  $A \in \As$.
\end{theorem}
\begin{proof}  First, 
assume that $f$ is  Gould-integrable with respect to  $U_M$, and 
denote by $j$ its integral. This means that the filtering net $(U(\sigma_M(f,P))_P$ is convergent to $j$. 
Hence it is Cauchy in $C(\Omega)$. Then, also the net $(\sigma_M(f,P))_P$ is Cauchy in $ck(X)$: by 
completeness of this space, 
 $(\sigma_M(f,P))_P$ has limit $J$ in $ck(X)$. By continuity of $U$, it is then clear that $U(J)=j$. \\
A similar argument can be used to prove the converse implication. So to conclude the proof it only remains to 
deduce integrability of $f$ in every subset $A\in \As$, and this is a consequence of integrability of $f$ with 
respect to $U(M)$: indeed, fixing any subset $A\in \As$ and any positive $\varepsilon$ in $\mathbb{R}$, a partition $P$ exists,
 finer than $\{A, T \setminus A\}$, such that
$\|\sigma_{U(M)}(f,P')-\sigma_{U(M)}(f,P'')\|_{\infty}\leq \varepsilon$
holds, for all partitions $P'$ and $P''$ finer than $P$. So,
choosing two partitions of $A$, say  $\Pi'_A$ and $\Pi''_A$, both finer than $P_A$ (i.e. $P$ restricted to $A$), and 
 {\em extending} them to $A^c$ with a unique partition finer than $P_{A^c}$, then two partitions, $P'$ and 
$P''$, can be found, both finer than $P$, and coincident in the set $A^c$: these partitions satisfy
$\varepsilon>\|\sigma_{U(M)}(f,P') -\sigma_{U(M)}(f,P'')\|_{\infty}=
\|\sigma_{U(M)}(f,\Pi'_A) -\sigma_{U(M)}(f,\Pi''_A)\|_{\infty}.$
By the completeness of $C(\Omega)$, this is enough to deduce integrability of $f\mathbf{1}_A$.
\end{proof} \\

Following Definition \ref{minteg}, 
 a multisubmeasure $M:\mathcal{A}\to ck(X)$,
it is said to be {\em integrable} if the  function $f(x)\equiv \mathbf{1}_A$ is 
$M$ integrable
for every $A\in \As$. Then the notation
$$M_0(A):=\int_T \mathbf{1}_A dM:=\int_A \mathbf{1} \,dM,$$ 
is used, for $A\in \mathcal{A}$.
This means that, for every element $A\in \mathcal{A}$ there exists an element $M_0(A)\in ck(X)$ such that,
 for every $\varepsilon>0$ a partition $P\in \Ps$ can be found with the property that
$h(\sum_{I\in P'}M(I\cap A),M_0(A))\leq \varepsilon$
holds, as soon as $P'$ is a finer partition than $P$.

The following theorem states a necessary and sufficient condition for the integrability of a $ck(X)$-valued 
fuzzy multisubmeasure of
bounded variation. The technique takes into account  previous results and it is inspired by the notions of 
\cite[Definition 3.4]{bcs2014}
and \cite[Definition 3.13]{bms}. This equivalence could be also useful in order to study differential inclusions.
\begin{theorem}\label{integrabilita}
Let $M$ be a fuzzy multisubmeasure of bounded variation. 
If $M$ is $ck(X)$-valued, then $M$ is integrable if and only if there exists a convex  compact set $K$ such
 that $\sigma(1,P):=\sigma_M (1,P) \subset K$
 for all partitions $P$. When this is the case, then
$$\int_T M =cl\left(\bigcup\{\sigma(1,P): P\in \mathcal{P}(T)\} \right).$$
\end{theorem}
\begin{proof} 
Sufficiency:  first of all, thanks to bounded variation, 
all the sums $\sigma(1,P):=\sum_{I\in P}M (I)$ are compact convex sets contained in $K$ for all partitions $P$.
 Moreover, thanks to subadditivity, the sums above are a filtering family in $ck(X)$.   In order to prove the 
existence of the integral, it is enough to show that the map $P\mapsto \sigma(1,P)$ is Cauchy, since $ck(X)$
is a complete space with respect to the Hausdorff distance.  Assume by contradiction that the Cauchy 
property does not hold: then there exists a positive number $\vp$ such that, as soon as $P$ is any partition of $T$,
 a couple $(P',P'')$ of finer partitions exists, satisfying $h(\sigma(1,P'),\sigma(1,P''))\geq \vp$.
 Since the refinement order is filtering, there exists a sequence $(P_n)_n$ of partitions, increasing in
 the refinement order, and such that $h(\sigma(1,P_n),\sigma(1,P_{n+1}))\geq \vp$ for all $n$. 
Now, since $\sigma(1,P_n)$ is an increasing sequence of elements of $ck(X)$, 
Proposition \ref{riparazione} applies, 
and the limit $\lim_n\sigma(1,P_n)$ exists, with respect to the Hausdorff distance: but this contradicts
 the fact that $h(\sigma(1,P_n),\sigma(1,P_{n+1}))\geq \vp$ for all $n$. So this part of the theorem is proved.
\\
Necessity: choose any partition $P$ of $T$, $P=\{E_1,...,E_k\}$, and fix arbitrarily $\vp>0$. 
By integrability, there exists a partition $P_{\vp}$ such that, for every finer partition $P'$ it holds
$h\left(\sigma(1,P'),\int_T M \right)\leq \vp,$
from which
$\sigma(1,P')\subset \int_T M+\vp B_X.$
Therefore, choosing $P':=P_{\vp}\vee P,$ and thanks to subadditivity of $M$, it follows
$\sigma(1,P)\subset \sigma(1,P')\subset  \int_T M+\vp B_X.$
By the arbitrariness of $P$, one gets
$$\bigcup_{P\in\mathcal {P}(T)}  \sigma(1,P)\subset \int_T M+\vp B_X,$$
where $P$ ranges over all the possible partitions of $T$. By the  arbitrariness of $\vp>0$ and 
compactness of $\int_T M$, it is obvious that
$\bigcup_{P\in\mathcal {P}(T)}  \sigma(1,P)\subset \int_T M,$
and so the necessity is proven.
Observe that, since $\int_T M$ is closed, the inclusion
$$cl\left(\bigcup_{P\in\mathcal {P}(T)}  \sigma(1,P) \right)\subset \int_T M$$
follows from the last formula. \\
In order to finish the proof, it only remains to prove the reverse inclusion, assuming that $M$ is Gould integrable.
 To this aim, fix $\vp>0$ and any partition $P$ of $T$ such that
$\int_T M\subset \sigma(1,P')+\vp B_X$
holds, for all partitions $P'$ finer than $P$. Then
$$\int_T M \subset cl\left(\bigcup_{P\in\mathcal {P}(T)}\sigma(1,P) \right)+ \vp B_X.$$
Then, by the arbitrariness of $\vp$, the desired conclusion follows.
\end{proof} \\
Now, assuming that $M$ is integrable,  the additivity  of  its {\em integral mapping} $M_0$ will be proven 
together with the equivalence between  the integrability of positive functions $f$  with respect to $M$ 
and  with respect to $M_0$, at least when $f$ is bounded. So, at least for bounded positive mappings $f$,
 integrability with respect to $M$ is equivalent to integrability with respect to an  additive multimeasure (with the same integral).\\

\begin{remark}\label{M-UM} \rm
Observe that, 
since the embedding theorem applies,  $M$ can be always identified with $U_M$, 
so that  $M$   can be viewed as 
a single-valued mapping, taking values in $C(\Omega)$ and so all the results 
concerning Gould integrability in Section \ref{tot+G} can be applied to $M$. 
\end{remark}
\begin{description}
\item[$H_0$)] From now on the fuzzy multisubmeasure $M: \As \to ck(X)$  will  always be assumed  of bounded variation and satisfying 
the conditions in Theorem \ref{integrabilita}, namely 
there exists a convex  compact set $K$ such that $\sigma(1,P)\subset K$
 for all partitions $P$,
so that $M$ is integrable in $ck(X)$. 
\end{description}
First of all, observe that, in those conditions, $\int_T \mathbf{1}_A\, dM$ exists, for every  set $A\in \As$,
 essentially with the same proof of Theorem \ref{integrabilita}. Next, 

\begin{proposition}\label{additivita}
The function $M_0$, defined in $\As$ as
$M_0(A)=\int_T \mathbf{1}_A\, dM$, is additive.
\end{proposition}
\begin{proof}
The result is a consequence  of Proposition \ref{additivita0}, and Theorem \ref{singlegould}.
\end{proof} \\

Now, notice that, under the above conditions, the $C(\Omega)$-valued measure $U_{M_0} - U_M$
 (which is non-negative, of course) has null integral. This almost immediately implies the following result.
\begin{theorem}\label{replace}
Let 
$f:T\to \mathbb{R}_0^+$ any bounded mapping. Then $f$ is $M$-integrable if and only if it is 
$M_0$-integrable, and the integrals coincide.
\end{theorem}
\begin{proof}
Assume that $f$ is integrable with respect to $M$. Then it is 
$U_M$-integrable, and 
$U\left(\int_T f dM \right)=\int_Tf dU_M.$
Now, it is sufficient to apply Corollary \ref{4.9} to deduce 
that $f$ is integrable also with respect to $U_{M_0}$ with the same integral, so
$U\left(\int_T f dM \right)=\int_T f dU_M=\int_T f dU_{M_0},$
which also shows that $\int_Tf dU_{M_0}$ is in the range of $U$. Thanks to Theorem \ref{singlegould}
 this is enough to conclude that $f$ is integrable with respect to $M_0$ and the two integrals agree.
A reverse argument also shows  the opposite implication, so the proof is complete.
\end{proof} \\

Also integrability on measurable subsets can be deduced, in the usual manner.
\begin{proposition}\label{questa}
Let
$f:T\to \mathbb{R}_0^+$ be any Gould-integrable mapping with respect to $M$. 
Then, if $A$ is any fixed element of $\As$, the mapping $f \mathbf{1}_A$ is integrable too.
\end{proposition}
\begin{proof} It is analogous to Proposition \ref{gouldsotto}.
\end{proof} \\

Moreover
\begin{proposition}\label{variazioneintegrale}
$M, M_0$ have the same variation measure. 
\end{proposition}
\begin{proof}
Fix arbitrarily $A\in \mathcal{A}$, and denote by $\mathcal{P}(A)$ the family of all finite partitions of $A$.
 Then, thanks to Theorem \ref{integrabilita},
$M_0(A)=cl(\bigcup\{\sum_i M(B_i) : (B_i)_i\in \mathcal{P}(A)\}).$
In particular, $M(A)\subset M_0(A)$, and, for any partition $(B_i)_i$ in $\mathcal{P}(A)$, one has
$\sum_i|M(B_i)|_h \leq \sum_i |M_0(B_i)|_h \leq v_{M_0}(A).$
This clearly implies that $v_M(A)\leq v_{M_0}(A)$.\\
 Conversely, fix any $\varepsilon>0$ and $A\in \mathcal{A}$. For every partition 
$P\equiv(B_i)_{i=1}^N \in \mathcal{P}(A)$, and every index $i$,
since $M$ is integrable, 
there exists a partition $P_i\equiv(B'_{i,j})_j \in \mathcal{P}(B_i)$ such that 
$|M_0(B_i)|_h \leq |\sum_j M(B'_{i,j}) |_h +\vp/N \leq \sum_j |M(B'_{i,j})|_h+\vp/N$
and so
$\sum_i |M_0(B_i)|_h \leq \sum_i \sum_j |M(B'_{i,j})|_h +\vp \leq v_M(A)+\vp.$
By the arbitrariness of $P\in \mathcal{P}(A)$ and of $\vp>0$, it follows 
$v_{M_0}(A)\leq v_M(A).$
 \end{proof} \\

In the sequel, 
assume that $\As$ is a $\sigma$-algebra.
Also let  $\Gamma:\mathcal{A}\to ck(X)$ be any fixed fuzzy multimeasure.
Following \cite{ref11,ms94b}
for $\alpha >0$ and $E \in \mathcal{A}$, let $A_{\Gamma} (E, \alpha) $ be the
 {\em $\alpha$-approximate range} defined by:
$A_{\Gamma} (E, \alpha) = \{ r\in [0,+\infty):  h(\Gamma(H), r M(H))\leq \alpha v_M(H), \forall H \in 
\mathcal{A} \cap E\}.$
\begin{remark}\label{rango-approx}\rm
Observe that, 
by Theorem \ref{LABU}, using the  embedding $U$, it is possible to formulate the $\alpha$-approximate
 range in the following way:
\begin{eqnarray*}
A_{\Gamma} (E, \alpha) &=& \left\{ r\in [0,+\infty):  
\| U_{\Gamma} (H) - r U_M (H) \|_{\infty} \leq \alpha \overline{m} (H), \forall H \in 
\mathcal{A} \cap E \right\},
\end{eqnarray*}
where $m:=U_M$. In fact $h(\Gamma(H), r M(H)) = \| U_{\Gamma} (H) - r U_M (H) \|_{\infty}$ and
$$\overline{m} (H) =
 \sup \sum_i \| U_M(E_i)\|_{\infty} = \sup \sum_i |M(E_i)|_h = v_M (H). $$
\end{remark}

\begin{theorem}{\rm (\cite[Lemma 3.3]{ms94b})} 
 Let 
$\Gamma \ll v_M$ (i.e. $\forall
\varepsilon>0$, $\exists \,\, \delta (\varepsilon)=\delta>0$ such that for every $E \in \mathcal{A}$ 
with $v_M (E)<\delta,$ it follows $v_{\Gamma}(E)<\varepsilon$).
Then, for every $\alpha >0$, the property $"A_{\Gamma}(E,\alpha)\neq \emptyset"$  is $v_M$-null difference.
\end{theorem}

\begin{theorem}
 Let $P$ be a $v_M$-null difference property such that $P$ is $v_M$-exhaustive on $T$. 
Then there exists  a $v_M$-exhaustion of $T$, 
$(B_{i})_{i}$, such that every $B_{i}$ has $P$ and $T = \bigcup\limits_{i}B_{i}$.
\end{theorem}
\begin{proof}
Since $P$ is $v_M$-exhaustive on $T$, there exists a $v_M$-exhaustion of $T$, denoted by 
$(E_{i})_{i\in I}$ , such that
every $E_{i}$ has $P.$ Thus
\begin{eqnarray}\label{7}
\forall \,\varepsilon >0, \exists \,\, n_{0}(\varepsilon) = n_{0} \in
\mathbb{N}\; \text{such that}\; v_M(T \backslash \bigcup\limits_{i=1}^{n_{0}} E_{i}) <\varepsilon.
\end{eqnarray}
Let $E_{0} = T \backslash \bigcup\limits_{i \in I} E_{i}$. From the previous inequality
it results $v_M(E_{0}) = 0.$
Let $(B_{i})_{i\in I}$ be the family of sets defined by: $B_{1} =
E_{0}\cup E_{1}\in \mathcal{A}, B_{i} = E_{i}\in \mathcal{A}$ for $i \geq 2$. 
Then $v_M(B_{1})= v_M(E_{1})>0$ since
$v_M(E_{1})\leq v_M(B_{1})\leq v_M(E_{1})+v_M(E_{0})= v_M(E_{1})$
 and
$v_M(B_{i}) = v_M(E_{i})>0$, for every $ i \geq 2.$
Obviously, $T = \cup_{i\in I} B_{i}$. 
It is 
$\cup_{i=1}^{n_{0}}B_{i} = E_{0}\cup (\cup_{i=1}^{n_{0}} E_{i})$.
Since $v_M(T \backslash \cup_{i=1}^{n_{0}}B_{i})\leq v_M(T\backslash \cup_{i=1}^{n_{0}}E_{i})<\varepsilon,$
then $(B_{i})_{i\in I}$ is a $v_M$-exhaustion of $T$. 
Now, for every $i \geq 2,\, B_{i} = E_{i}$ has $P$. So, 
it only remains  to prove that $B_{1}$ has $P.$ By the relations:
\[
B_{1}\triangle E_{1}= (E_{0}\cup E_{1})\triangle E_{1} =
E_{0}\backslash E_{1}\subset E_{0} \Rightarrow 0\leq \dd
v_M(B_{1}\triangle E_{1})\leq v_M(E_{0}) = 0,
\]
it follows that $v_M(B_{1}\triangle E_{1}) = 0$. Since
$P$ is $v_M$-null difference and $E_{1}$ has $P$, one concludes that $B_{1}$ has $P$.
\end{proof} \\

\begin{definition}\rm
A multimeasure $\G:\mathcal{A}\to 
ck(X)$ is said to be \textit{ dominated by} 
 $M$ if there is a constant $b>0$ such that $|\G(E)|_h \leq b v_M(E)$, for every $E \in \mathcal{A}$.
\end{definition}

\begin{lemma}{\rm (see \cite[Lemma 2.9]{ms-trieste} for semivariation)}\label{due}.
For every $E\in\mathcal{A}$ with $v_M(E)>0$, there exists
$B\in\mathcal{A} \cap E$,  
such that $v_M(B)<2|M(B)|_h$.
\end{lemma}
\begin{proof}
By contradiction, there exists an element $E\in \mathcal{A}$, with $v_M(E)>0$ such that
 $v_M(B)\geq 2|M(B)|_h$ for all measurable sets $B\subset E$, with $v_M(B)>0$. 
Fix $\varepsilon >0$ arbitrarily and pick a disjoint family $B_1,...,B_k$ of measurable subsets of $E$,
 such that $v_M(B_i)>0$ for all $i$ and
$\sum_i|M(B_i)|_h+\varepsilon \geq v_M(E).$
Since $v_M$ is additive, from the contradiction assumed we obtain
$v_M(E)=\sum_iv_M(B_i)\geq 2\sum_i |M(B_i)|_h \geq 2v_M(E)-2\varepsilon.$
Since $\varepsilon$ is arbitrarily small it follows that $v_M(E)\leq 0$, giving a contradiction.  
\end{proof} \\

We state now our main theorem,
which is inspired by  \cite[Lemma 3.1]{ms-trieste}. We point out, however, that our result is
 new even in the single-valued case, since additivity is requested just from one of the measures involved.
\begin{theorem}{\rm [Radon Nikod\'{y}m]}\label{trn}
Let $M$ and $\Gamma$ be  $ck(X)$-valued fuzzy multisubmeasures,
satisfying the  condition  $H_0$.
 Suppose moreover that  $\Gamma$ is additive and
\begin{enumerate}[\rm \bf \ref{trn}.1)]
\item  $\Gamma$ is 
dominated by $M;$
\item for every $\varepsilon>0,$ the set property $"A_{\Gamma} ( E, \varepsilon)\neq
\emptyset$" is $v_M$-exhaustive on every $E \in \mathcal{A}$.
\end{enumerate}
Then there exists an $M$-integrable bounded function 
$f: T\to \mathbb{R}_0^+$ such that for every $E \in\mathcal{A}$ it is
$\Gamma(E) = \int_{E} f dM.$
\end{theorem}
\begin{proof}
Thanks to Theorems \ref{singlegould} and  \ref{replace}, it will be sufficient to prove that, for the 
single-valued mappings $U_{\Gamma}$ and $U_{M_0}$, there exists a bounded measurable
 Radon-Nikod\'{y}m derivative $f$. Indeed,
since $\Gamma$ is dominated by  $M$,
the same property  holds with respect to $M_0$, since $v_M=v_{M_0}$ by Proposition \ref{variazioneintegrale}.
Of course, then the single-valued additive mapping $U_{\Gamma}$ is 
dominated by  $U_{M_0}$. \\
Now, it will be proved that $\Gamma$ and $M_0$ satisfy the condition of exhaustivity of the property 
$``A_{\Gamma}(E,\varepsilon)\neq \emptyset".$ 
To this aim, the notation $A_{\Gamma,M}(E,\varepsilon)$ will be used, in order to stress the role of 
$M$ in $A_{\Gamma}(E,\varepsilon)$.  Then, assuming that the set $A_{\Gamma,M}(E,\varepsilon)$
 is nonempty, it also turns out that 
$A_{\Gamma,M_0}(E,\varepsilon)$ is nonempty. Indeed, let $\sigma>0$ and $r\in A_{\Gamma,M}(E,\varepsilon)$
 be fixed. Then, for every measurable $H\subset E$,  there exists a finite partition $\{H_1,...,H_l\}$ of $H$,
such that 
$h(rM_0(H),r\sum_{i=1}^lM(H_i))\leq \sigma.$ 
Now, one has
$h(\Gamma(H),r M_0(H)) \leq h(\Gamma(H),r\sum_{i=1}^lM(H_i))+\sigma=
 h(\sum_{i=1}^l\Gamma(H_i), r\sum_{i=1}^lM(H_i))+\sigma\leq
 \sum_{i=1}^l h(\Gamma(H_i),rM(H_i))+\sigma\leq  \varepsilon \sum_{i=1}^lv_M(H_i)+\sigma=
\varepsilon v_M(H)+\sigma=\varepsilon v_{M_0}(H)+\sigma.
$
Since $\sigma$ is arbitrary, this shows that $r\in A_{\Gamma,M_0}(E,\varepsilon)$, and, in turn, this implies the exhaustivity 
of the property 
$``A_{\Gamma}(E,\varepsilon)\neq \emptyset".$ also with respect to $M_0$.\\
Of course, since $U$ is an isometric embedding,  the measures $U(\Gamma)$ and $U(M_0)$ also enjoy the 
same properties of absolute continuity and exhaustivity. 
Once a bounded measurable Radon-Nikod\'{y}m derivative $f$ has been found, for $U_{\Gamma}$ with respect to $U_{M_0}$, 
then one has integrability of $f$ with respect to $M_0$ by the Theorem \ref{singlegould} and
$U\left(\int_E f d M_0 \right)=\int_E f\, dU_{M_0}=U_{\Gamma}(E),$
while
$\int_E f dM_0=\int_E f dM, $
holds true for every measurable set $E$ thanks to Theorem \ref{replace}.\\
 So, for all $E\in \As$ one has
$U\left(\int_E f dM  \right)=U_{\Gamma}(E)$
and therefore $\int_E f dM=\Gamma(E).$
So, the problem is to find a bounded measurable mapping $f$, derivative of $U_{\Gamma}$ with respect to $U_M$.
In order to prove this we can observe that 
 $U_M$ satisfies the assumption of \cite[Lemma 3.1]{ms-trieste}  and so an integrable function $f$ in the 
sense of \cite{ms-trieste} can be found. Finally \cite[Theorem 2.16]{gp} guarantees the integrability in
 the present sense.
\end{proof} \\

Observe that Theorem \ref{trn} extends theorems given in \cite{ms-trieste,ms94b}, moreover
as an application 
 we can consider a  Gould integrable multifunction  (in the sense of \cite[Definition 16]{pap2016}) $\varphi: T \to ck(X)$
with respect to a probability $\mu$; let $M$ be its Gould integral:  $M: \mathcal{A} \to ck(X)$.
Thanks to \cite[Theorem 11]{pap2016} the set valued function $M$ is additive, and satisfies the condition $H_0$.
Observe moreover that, thanks to the R{\aa}dstr\"{o}m's embedding and \cite[Definition 16]{pap2016},  an analogous version of 
Corollary \ref{henstoc2}
 holds when the measure involved is scalar and the integrand is $C(\Omega$)-valued. 
Then we have
\begin{corollary}\label{cor-finale}
Let $\Gamma$ be any $ck(X)$-valued fuzzy multimeasure, satisfying $H_0$, 
 {\rm \ref{trn}.1)} and {\rm \ref{trn}.2)}. Then there exists a scalar
$M$-integrable bounded mapping $f:T\to \erre^+$ such that, for every $A\in \mathcal{A}$,
\begin{eqnarray}\label{eq:sostituzione}
\Gamma(A)=\int_A f(t)\varphi(t)d\mu.
\end{eqnarray}
\end{corollary}
\begin{proof}
Thanks to Theorem \ref{trn}, there exists a bounded  $M$-integrable mapping $f:T\to \erre^+$, such that, for all $A\in \mathcal{A}$,
$\int_A f(t) dM=\Gamma(A).$
By \cite[Theorem 5]{pap2016} it is enough to prove the assertion for $A = T$. Without loss of generality we also can consider
 $\varphi, M, \Gamma$ as objects with values in  $C(\Omega)$, as pointed out in Remark \ref{M-UM} and in Theorem 
\ref{LABU}. In this new setting, we can use  Corollary \ref{henstoc2}
(in both versions: scalar functions and $C(\Omega$)-valued measures and viceversa)
 and so,
 there exists a sequence of  partition $(P_n)_n$ in $T$ such that, for every $n \in \mathbb{N}$ and 
for every partition $P$  finer than $P_n$, one has:
$ \sum_{J \in P} | M(J) - \varphi (t_J) \mu(J) | \leq \dfrac{u}{n}, 
\,  \sum_{J \in P} | \Gamma(J) - f (t_J) M(J) | \leq \dfrac{u}{n} 
$
for every choice of $t_J \in J$. Then, for every $P$ finer than $P_n$,
\begin{eqnarray*}
&& | \sum_{J \in P}  f(t_j) \varphi (t_J) \mu(J) - \Gamma(T)| \leq \sum_{J \in P} |  f(t_j) \varphi (t_J) \mu(J) - \Gamma(J)| \leq\\
&&
\leq \sum_{J \in P} |  f(t_j) \varphi (t_J) \mu(J) - f(t_j)M(J)| +  \sum_{J \in P} | f(t_J) M(J) - \Gamma(J)| \leq \\ 
&&
\leq \sum_{J \in P} |  f(t_j)| \cdot | \varphi (t_J) \mu(J) - M(J)| +  \dfrac{u}{n} \leq  \dfrac{u}{n} (1 + \sup |f|).
\end{eqnarray*}
So $(\sigma(f \varphi, P))_P$ is convergent in $C(\Omega)$ to $\Gamma(T)$, and so $f\varphi$ is $\mu$-integrable and
$\Gamma(T)=\int_T f(t)\varphi(t)d\mu.$
\end{proof} \\
The last result of Corollary \ref{cor-finale} can  be viewed as
 an integration by substitution for fuzzy multisubmeasures and 
(\ref{eq:sostituzione}) can be written as
$$ \displaystyle{\int_A} f dM = \displaystyle{\int_A} f \varphi d\mu, \quad \forall \, A \in \As.$$

\section*{Conclusions}
We have studied the Gould-integrability of a scalar function $f$ with respect to a set-valued, non necessarily additive measure $m$.
 In particular we have focused our attention on compact and convex-valued measures. In this case, thanks to the well-known
  R{\aa}dstr\"{o}m's embedding theorem, $m$ can be considered as a measure taking values in the Banach lattice $C(\Omega)$. 
In addition, the  notion of {\em integrability} has been  introduced for $m$, with the purpose  to avoid the requirement of additivity.  
In fact, thanks to the R{\aa}dstr\"{o}m's embedding, we are able to establish a Henstock-type theorem for this kind of integral.
This, in turn, implies that any {\em integrable} measure $m$ can be seen as an additive measure, plus a {\em negligible} one. 
Finally a Radon-Nikod\'{y}m Theorem is obtained in this situation which is new also in the finite dimensional case, since one of the involved set-valued
 measures is non additive.
Moreover,  a  u-substitution result for fuzzy multimeasures is established.


\small

\footnotesize
\Addresses
\end{document}